\documentclass[12pt,reqno]{amsart}

\newcommand\version{October 9, 2015}

\usepackage{amsmath,amsfonts,amsthm,amssymb,amsxtra}

\newtheorem{theorem}{Theorem}[section]
\newtheorem{proposition}[theorem]{Proposition}
\newtheorem{lemma}[theorem]{Lemma}

\theoremstyle{definition}

\theoremstyle{remark}

\numberwithin{equation}{section}

\newcommand{\C}{\mathbb{C}}

\renewcommand{\epsilon}{\varepsilon}
\newcommand{\ess}{\mathrm{ess}}

\newcommand{\N}{\mathbb{N}}

\renewcommand{\phi}{\varphi}
\newcommand{\R}{\mathbb{R}}

\DeclareMathOperator{\codim}{codim}
\DeclareMathOperator{\dist}{dist}

\DeclareMathOperator{\dom}{dom}
\DeclareMathOperator{\im}{Im}
\DeclareMathOperator{\ind}{ind}
\DeclareMathOperator{\ran}{ran}
\DeclareMathOperator{\re}{Re}

\DeclareMathOperator{\tr}{Tr}

\begin{document}

\title[Schr\"odinger operators with complex potentials --- \version]{Eigenvalue bounds for Schr\"odinger operators\\ with complex potentials. III}

\author{Rupert L. Frank}
\address{Rupert L. Frank, Mathematics 253-37, Caltech, Pasadena, CA 91125, USA}
\email{rlfrank@caltech.edu}

\begin{abstract}
We discuss the eigenvalues $E_j$ of Schr\"odinger operators $-\Delta+V$ in $L^2(\R^d)$ with complex potentials $V\in L^p$, $p<\infty$. We show that (A) $\re E_j\to\infty$ implies $\im E_j\to 0$, and (B) $\re E_j\to E\in [0,\infty)$ implies $(\im E_j)\in\ell^q$ for some $q$ depending on $p$. We prove quantitative versions of (A) and (B) in terms of the $L^p$-norm of $V$.
\end{abstract}


\maketitle

\renewcommand{\thefootnote}{${}$} \footnotetext{\copyright\, 2015 by the author. This paper may be reproduced, in its entirety, for non-commercial purposes.}

\section{Introduction and main results}

In this paper we continue our study \cite{FrLaLiSe,Fr,FrSa,FrSi} of eigenvalues of Schr\"odinger operators $-\Delta+V$ in $L^2(\R^d)$ with complex-valued potentials. We are interested in quantitative information about the location of the eigenvalues under the sole assumption that $V\in L^p(\R^d)$ for some $p<\infty$. We also need $p\geq 1$ if $d=1$, $p>1$ if $d=2$ and $p\geq d/2$ if $d\geq 3$ in order to define $-\Delta+V$ as an $m$-sectorial operator. As usual in connection with Lieb--Thirring inequalities we will write $p=\gamma+d/2$ in the following. While we treated the `short range' case $\gamma\leq 1/2$ in \cite{Fr,FrSa,FrSi}, here we will be mostly concerned with the `long range' case $\gamma>1/2$.

It is an easy consequence of relative form compactness that under the above assumptions on $\gamma$, the spectrum of $-\Delta+V$ in $L^2(\R^d)$ in $\C\setminus[0,\infty)$ consists of isolated eigenvalues of finite algebraic multiplicities; see Proposition \ref{spec} in the appendix.

Our first topic is the location of individual eigenvalues. In \cite{Fr} (see also \cite{FrSi}) we have shown that, if $\gamma\leq 1/2$, all eigenvalues lie in a disk whose radius is controlled by the $L^{\gamma+d/2}$-norm of $V$. More precisely, any eigenvalue $E\in\C$ of $-\Delta+V$ satisfies
\begin{equation}
\label{eq:davies}
|E|^\gamma \leq D_{\gamma,d} \int_{\R^d} |V|^{\gamma+d/2}\,dx
\end{equation}
for $\gamma=1/2$ if $d=1$ and $0<\gamma\leq 1/2$ if $d\geq 2$ and with a constant $D_{\gamma,d}$ independent of $V$. This bound for $d=1$ is due to Abramov, Aslanyan and Davies \cite{AAD}; see also \cite{DN,FrSi}. Our first main result in this paper is a replacement of this bound for $\gamma>1/2$. We will use the notation
$$
\delta(z) := \dist(z,[0,\infty)) \,.
$$

\begin{theorem}
\label{main1}
Let $d\geq 1$ and $\gamma\geq 1/2$. Then any eigenvalue $E$ of $-\Delta+V$ in $L^2(\R^d)$ satisfies
\begin{equation}
\label{eq:main1}
\delta(E)^{\gamma-1/2} |E|^{1/2} \leq D_{\gamma,d} \int_{\R^d} |V|^{\gamma+d/2}\,dx \,.
\end{equation}
\end{theorem}

\emph{Remarks.} (1) Using $|E|\geq |\re E|$ and the fact that $\delta(E)=|\im E|$ if $\re E\geq 0$, we obtain
\begin{equation}
\label{eq:main1cor}
|\im E| \leq D_{\gamma,d}^\frac{1}{\gamma-1/2} (\re E)^{-\frac{1/2}{\gamma-1/2}} \left( \int_{\R^d} |V|^{\gamma+d/2}\,dx \right)^{\frac{1}{\gamma-1/2}} 
\end{equation}
for eigenvalues $E$ with $\re E\geq 0$. This bound has the following \emph{important consequence} (which we have not found in the literature): If $(E_j)$ is a sequence of eigenvalues of $-\Delta+V$ in $L^2(\R^d)$ with $\re E_j\to\infty$, then $\im E_j\to 0$ (provided $V\in L^{\gamma+d/2}(\R^d)$ with $\gamma> 1/2$).\\
(2) Theorem \ref{main1} improves upon a recent result by Enblom \cite{En} who showed that $\delta(E)^\gamma \leq D'_{\gamma,d} \int |V|^{\gamma+d/2}\,dx$ (which, in turn, improves upon earlier results in \cite{FrLaLiSe,LaSa}). Note that Enblom's result does \emph{not} imply the consequence stated in (1).\\
(3) Theorem \ref{main1} is, in fact, a simple consequence of the method of proof of the bound \eqref{eq:davies} with $\gamma=1/2$ and the trival bound $\delta(E)\leq \|V\|_\infty$ corresponding to $\gamma=\infty$. When $d=1$, this argument yields \eqref{eq:main1} with the explicit constant
$$
D_{\gamma,1} = 2^{-1} \,;
$$
see the remark following the proof of Theorem \ref{res}.\\
(4) For $d=1$ the bounds \eqref{eq:main1} and \eqref{eq:main1cor} are closely related to the bound
\begin{equation}
\label{eq:dn}
|E|^{\frac{\gamma+1/2}{2}} (\re\sqrt{-E})^{\gamma-1/2} \leq \frac12 \left( \frac{\gamma-1/2}{\gamma+1/2} \right)^{\gamma-1/2} \int_{\R} |V|^{\gamma+1/2}\,dx
\end{equation}
from \cite[Cor. 2.17]{DN} (with the branch of the square root on $\C\setminus(-\infty,0]$ having positive real part). This bound coincides with \eqref{eq:main1cor} for $\re E \gg|\im E|$, up to the value of the constant, because then $\re\sqrt{-E}\sim |\im E|/(2\sqrt{\re E})$. However, \eqref{eq:dn} is proved using real interpolation while we obtain \eqref{eq:main1} using complex interpolation.

\medskip

Our second topic is the distribution of eigenvalues as quantified by analogues of the Lieb--Thirring bounds \cite{LT}. Quantitative information for $\gamma\leq 1/2$ was obtained in \cite{FrSa}, where it was shown that the eigenvalues $E_j\in\C\setminus[0,\infty)$ of $-\Delta +V$, repeated according to their algebraic multiplicities, satisfy
\begin{equation}\label{eq:LTfrsa1}
\left( \sum_j \frac{\delta(E_j)}{|E_j|^{1/2}} \right)^{2\gamma} \leq L_{\gamma,d,\epsilon}\, \int_{\R^d} |V|^{\gamma+d/2}\,dx
\qquad\text{if}\ \gamma<\frac{d}{2(2d-1)} \ \text{and}\ d\geq 2 \,,
\end{equation}
and
\begin{align}\label{eq:LTfrsa2}
\left( \sum_j \frac{\delta(E_j)}{|E_j|^{\frac{d(1/2-\gamma)}{d/2-\gamma} - \frac{\epsilon}{2}}} \right)^\frac{2\gamma(d-2\gamma)}{4(d-1)\gamma+\epsilon(d-2\gamma)} \leq \ & L_{\gamma,d,\epsilon}\, \int_{\R^d} |V|^{\gamma+d/2}\,dx \notag \\
& \text{if}\ \frac{d}{2(2d-1)}\leq\gamma\leq\frac12 \ \text{and}\ d\geq 1 \,.
\end{align}
(In \eqref{eq:LTfrsa2} for $d=1$ and $\gamma=1/2$, we interpret the exponent $d(1/2-\gamma)/(d/2-\gamma)$ as $0$ and the exponent $2\gamma(d-2\gamma)/(4(d-1)\gamma+\epsilon(d-2\gamma))$ as $2\gamma/(2+\epsilon)$.) Our second main result is a replacement of these bounds for $\gamma>1/2$.

\begin{theorem}\label{main2}
Let $d\geq 1$ and $\gamma>1/2$. Then the eigenvalues $E_j\in\C\setminus[0,\infty)$ of $-\Delta+V$ in $L^2(\R^d)$, repeated according to their algebraic multiplicity, satisfy for any $\epsilon>0$,
\begin{equation}\label{eq:main2}
\left( \sum_{|E_j|^\gamma\leq C_{\gamma,d} \int |V|^{\gamma+d/2}\,dx} \delta(E_j)^{2\gamma+\epsilon} \right)^\frac{\gamma}{2\gamma+\epsilon} \leq L_{\gamma,d,\epsilon} \int_{\R^d} |V|^{\gamma+d/2}\,dx
\end{equation}
and for any $\epsilon>0$, $0<\epsilon'<\gamma/(\gamma+d/2)$ and $\mu\geq 1$,
\begin{equation}\label{eq:main21}
\left( \sum_{|E_j|^\gamma\geq \mu C_{\gamma,d} \int |V|^{\gamma+d/2}\,dx} \frac{\delta(E_j)^{2\gamma+\epsilon}}{|E_j|^{2\gamma+\epsilon - \frac{\gamma}{\gamma+d/2} + \epsilon'}} \right)^{\frac{\gamma(\gamma+d/2)}{\gamma- \epsilon'(\gamma+d/2)}} \leq L_{\gamma,d,\epsilon,\epsilon'} \mu^{-\frac{\epsilon'(\gamma+d/2)}{\gamma-\epsilon'(\gamma+d/2)}} \int_{\R^d} |V|^{\gamma+d/2}\,dx \,.
\end{equation}
\end{theorem}

Using a different technique we can improve the bound on the accumulation for $\gamma\geq d/2$.

\begin{theorem}\label{main3}
Let $\gamma\geq 1/2$ if $d=1$ and $\gamma>0$ if $d\geq 2$. Then the eigenvalues $E_j\in\C\setminus[0,\infty)$ of $-\Delta+V$ in $L^2(\R^d)$, repeated according to their algebraic multiplicity, satisfy
\begin{equation}\label{eq:main3}
\left( \sum_{|E_j|^\gamma\leq C_{\gamma,d} \int |V|^{\gamma+d/2}\,dx} \delta(E_j)^{\gamma+d/2} \right)^\frac{\gamma}{\gamma+d/2} \leq L_{\gamma,d} \int_{\R^d} |V|^{\gamma+d/2}\,dx
\end{equation}
and for any $0<\epsilon'<\gamma$ and $\mu\geq 1$,
\begin{equation}\label{eq:main31}
\left( \sum_{|E_j|^\gamma\geq \mu C_{\gamma,d} \int |V|^{\gamma+d/2}\,dx} \frac{\delta(E_j)^{\gamma+d/2}}{|E_j|^{d/2 + \epsilon'}} \right)^{\frac{\gamma}{\gamma-\epsilon'}} \leq L_{\gamma,d,\epsilon,\epsilon'} \mu^{-\frac{\epsilon'}{\gamma-\epsilon'}} \int_{\R^d} |V|^{\gamma+d/2}\,dx \,.
\end{equation}
\end{theorem}

\emph{Remarks.} (1) These two theorems have the following \emph{consequence}: if $(E_j)$ is a sequence of eigenvalues of $-\Delta+V$ in $L^2(\R^d)$ with $E_j\to E\in [0,\infty)$, then $(\im E_j)\in\ell^p$ for every $p>2\gamma$ if $1/2<\gamma<d/2$ and for $p=\gamma+d/2$ if $\gamma\geq d/2$, provided $V\in L^{\gamma+d/2}(\R^d)$.\\ 
(2) We particularly highlight the bound of Theorem \ref{main3} for $d=1$ and $\gamma=1/2$. Combined with \eqref{eq:davies} it implies that
$$
\left( \sum_j \delta(E_j) \right)^{1/2} \leq L_{1/2,1}' \int_\R |V| \,dx \,.
$$
(3) We stress that in Theorems \ref{main2} and \ref{main3} we distinguish between eigenvalues lying in a disk around the origin (whose radius is controlled in a scale-invariant way by a power of a norm of the potential) and eigenvalues lying outside this disk (which are constraint by \eqref{eq:main1cor} in terms of the norm of the potential).\\
(4) The proof of Theorem \ref{main2} is an extension of the proof of \eqref{eq:LTfrsa1} and \eqref{eq:LTfrsa2} in \cite{FrSa}. It is based on the analysis of zeroes of a regularized determinant. This method was pioneered in \cite{DeKa,BGK} and we use a remarkable theorem about zeroes of analytic functions due to Borichev, Golinskii and Kupin \cite{BGK}. The technical new ingredients are resolvent bounds in trace ideals related to the limiting absorption principle.\\
(5) The proof of Theorem \ref{main3} is based on a matrix inequality due to Hansmann \cite{Ha}, which extends a classical result of Kato. Inequalities similar to those in Theorem \ref{main3} were obtained in \cite{Ha,DeHaKa}, but they depend on a lower bound on $\re V$. We remove this dependence. Hansmann (private communication) has shown that the qualitative fact mentioned in (1) with $p=\gamma+d/2$ does not require $\re V$ to be bounded from below, and our Theorem \ref{main3} makes this quantitative.\\
(6) For the reader's convenience we summarize the results from \cite{DeHaKa} in Appendix \ref{sec:dhk} and compare them with our results. We will see that Theorems \ref{main2} and \ref{main3} are stronger in many situations.

\medskip

In Appendices \ref{sec:defop} and \ref{sec:zeroes} we provide proofs of `well-known' facts whose proofs we have not been able to find in the literature. They concern the spectrum of form compact perturbations of non-negative, self-adjoint operators and zeroes of regularized determinants.

\subsection*{Notation} The Schatten space with exponent $1\leq p<\infty$ is denoted by $\mathfrak S^p$ and its norm by $\|\cdot\|_p$, that is, the $\ell^p$-norm of the sequence of singular values. For $n\in\N$ we denote by $\det_n$ the $n$-th regularized determinant; see, e.g., \cite{Si}. We write $\sqrt{V(x)} =V(x)/\sqrt{|V(x)|}$ if $V(x)\neq 0$ and $\sqrt{V(x)}=0$ if $V(x)=0$. 

\subsection*{Acknowledgements} Fundamental for several of the new theorems here are results from \cite{FrSa}, which were obtained jointly with J. Sabin to whom the author is most grateful. He would also like to thank M. Demuth and M. Hansmann for fruitful discussions. This paper has its origin at the conference `Mathematical aspects of physics with non-self-adjoint operators' in June 2015 and the author is grateful to the organizers and the American Institute of Mathematics for the invitation. This paper was finished at the Mittag--Leffler Institute and the author is grateful to A. Laptev for the hospitality. Support through NSF grant DMS--1363432 is acknowledged.


\section{Resolvent bounds and the proof of Theorem \ref{main1}}

The proof of \eqref{eq:davies} in \cite{Fr} relies on bounds on the operator norm of $W_1(-\Delta-z)^{-1}W_2$ for $W_1,W_2\in L^{2\gamma+d}$ for $\gamma\leq 1/2$. Similarly, the proof of Theorem \ref{main1} will rely on such bounds for $\gamma\geq 1/2$. In fact, as shown in \cite{FrSa}, the operator norm bounds for $\gamma\leq 1/2$ can be improved to bounds in trace ideals, and here we will also prove the corresponding bounds for $\gamma\geq 1/2$ directly in trace ideals.

\begin{proposition}\label{res}
Let $d\geq 1$ and $\gamma\geq 1/2$. Then there is a constant $C$ such that for all $W_1,W_2\in L^{2(\gamma+d/2)}(\R^d)$ and all $z\in\C\setminus[0,\infty)$,
\begin{equation}
\label{eq:res}
\| W_1 (-\Delta-z)^{-1} W_2 \|_{2(\gamma+d/2)} \leq C \delta(z)^{-1+\frac{(d+1)/2}{\gamma+d/2}} |z|^{-\frac{1}{2(\gamma+d/2)}} \|W_1\|_{2(\gamma+d/2)} \|W_2\|_{2(\gamma+d/2)} \,.
\end{equation}
\end{proposition}

\begin{proof}
The result follows by complex interpolation between the bound for $\gamma=1/2$ proved in \cite{FrSa},
\begin{equation}
\label{eq:frsa}
\| W_1 (-\Delta-z)^{-1} W_2 \|_{d+1} \leq C |z|^{-\frac{1}{d+1}} \|W_1\|_{d+1} \|W_2\|_{d+1} \,,
\end{equation}
and the trivial bound for $\gamma=\infty$,
\begin{equation}
\label{eq:trivial}
\| W_1 (-\Delta-z)^{-1} W_2 \|_\infty \leq \delta(z)^{-1} \|W_1\|_\infty \|W_2\|_\infty \,.
\end{equation}
To carry out the details of the interpolation step we fix $\gamma>1/2$, bounded functions $W_j$ of compact support, a finite rank operator $K$ on $L^2(\R^d)$ as well as $z\in\C\setminus[0,\infty)$. We write $K=U|K|$ with a partial isomety $U$ and $W_j=J_j |W_j|$ with $|J_j|\leq 1$. For $\zeta\in\C$ with $0\leq\re\zeta\leq 1$ we define
$$
f(\zeta) := \tr J_1 |W_1|^{\zeta/\theta} (-\Delta-z)^{-1} |W_2|^{\zeta\theta} J_2 U |K|^{(d+1-\zeta)/(d+1-\theta)}
$$
with $\theta:= (d+1)/(d+2\gamma)$. Then $f$ is a bounded continuous function in $\{0\leq\re\zeta\leq 1\}$ which is analytic in its interior. Moreover, if $t\in\R$, by \eqref{eq:trivial},
$$
|f(it)| \leq \delta(z)^{-1} \tr |K|^{(d+1)/(d+1-\theta)} = \delta(z)^{-1} \|K\|_{(d+1)/(d+1-\theta)}^{(d+1)/(d+1-\theta)}
$$
and, by \eqref{eq:frsa},
\begin{align*}
|f(1+it)| & \leq C |z|^{-\frac{1}{d+1}} \||W_1|^{1/\theta}\|_{d+1} \||W_2|^{1/\theta}\|_{d+1} \||K|^{d/(d+1-\theta)}\|_{(d+1)/d} \\
& = C |z|^{-\frac{1}{d+1}} \|W_1\|_{(d+1)/\theta}^{1/\theta} \|W_2\|_{(d+1)/\theta}^{1/\theta} \|K\|_{(d+1)/(d+1-\theta)}^{d/(d+1-\theta)} \,.
\end{align*}
Thus, by Hadamard's three line lemma,
\begin{equation}
\label{eq:hadamard}
|f(\theta)| \leq C^\theta |z|^{-\frac{\theta}{d+1}} \delta(z)^{-1+\theta} \|K\|_{(d+1)/(d+1-\theta)} \|W_1\|_{(d+1)/\theta} \|W_2\|_{(d+1)/\theta} \,.
\end{equation}
By duality and by recalling the definition of $\theta$, this yields the bound in the proposition. Finally, by a density argument the bound is extended to all $W_j\in L^{2(\gamma+d/2)}(\R^d)$.
\end{proof}

\emph{Remark.}
If $d=1$ we can use the explicit form of the integral kernel of $(-\Delta-z)^{-1}$ to see that \eqref{eq:frsa} holds with $C=1/2$. Thus, in this case \eqref{eq:hadamard} yields the explicit bound
\begin{equation*}
\| W_1 (-\Delta-z)^{-1} W_2 \|_{2(\gamma+1/2)} \leq 2^{-\frac{1}{\gamma+1/2}} \delta(z)^{-1+\frac{1}{\gamma+1/2}} |z|^{-\frac{1}{2(\gamma+1/2)}} \|W_1\|_{2(\gamma+1/2)} \|W_2\|_{2(\gamma+1/2)} \,.
\end{equation*}

\bigskip

We are now in position to prove Theorem \ref{main1}. We use the same strategy as in \cite{AAD,Fr}.

\begin{proof}[Proof of Theorem \ref{main1}]
If $E\in\C\setminus[0,\infty)$ is an eigenvalue of $-\Delta+V$, then $-1$ is an eigenvalue of the Birman--Schwinger operator $\sqrt{V}(-\Delta-E)^{-1}\sqrt{|V|}$. Thus,
$$
1 \leq \left\| \sqrt{V} (-\Delta-E)^{-1} \sqrt{|V|} \right\| \leq \left\| \sqrt{V} (-\Delta-E)^{-1} \sqrt{|V|} \right\|_{2(\gamma+d/2)} \,.
$$
According to Proposition \ref{res}, the right side is not greater than
\begin{align*}
& C_{\gamma,d}\, \delta(E)^{-1+\frac{(d+1)/2}{\gamma+d/2}} |E|^{-\frac{1}{2(\gamma+d/2)}} \|\sqrt{V} \|_{2(\gamma+d/2)} \|\sqrt{|V|}\|_{2(\gamma+d/2)} \\
& \qquad = C_{\gamma,d}\, \delta(E)^{-1+\frac{(d+1)/2}{\gamma+d/2}} |E|^{-\frac{1}{2(\gamma+d/2)}} \|V \|_{\gamma+d/2} \,.
\end{align*}
This yields the bound of the theorem.
\end{proof}

\emph{Remark.} It is clear from the previous proof that for Theorem \ref{main1} we do not need the full strength of Proposition \ref{res}. A bound for the operator norm on the left side of \eqref{eq:res} would suffice. In order to prove such a bound one can replace \eqref{eq:frsa} by the uniform resolvent bounds of Kenig--Ruiz--Sogge \cite{KRS} (see also \cite{Fr}) and use again complex interpolation. We shall, however, need the trace ideal bounds in the proof of Theorem~\ref{main2}.


\section{A quantitative version of the analytic Fredholm alternative}\label{sec:af}

The analytic Fredholm alternative says that if $K(z)$, $z\in\Omega$, is an analytic family of compact operators on some connected open set $\Omega\subset\C$, then the set $\Sigma=\{z\in\Omega:\ -1\ \text{is an eigenvalue of}\ K(z)\}$ is at most countable and has no accumulation point in $\Omega$ (unless it coincides with $\Omega$). In this section we estimate the possible accumulation of $\Sigma$ at $\partial\Omega$ in the special case that $\Omega=\C\setminus[0,\infty)$ and under the assumption that $K(z)$ belongs to some trace ideal.

We begin by reviewing some facts about analytic families of operators and we refer, for instance, to \cite{GGK} for more details. Let $\Omega\subset\C$ be an open set and let $W(z)$, $z\in\Omega$, be an analytic family of bounded operators. A number $z_0\in\Omega$ is called an \emph{eigenvalue of finite type} of $W$, if $\ker W(z_0)\not =\{0\}$, if $W(z_0)$ is Fredholm (that is, both $\dim \ker W(z_0)$ and $\mathrm{codim} \ran W(z_0)$ are finite) and if $W(z)$ is invertible for all $0<|z-z_0|<\epsilon$ for some $\epsilon>0$. It follows from \cite[Thm.$\!$ XI.8.1]{GGK} that, if $z_0$ is an eigenvalue of finite type of $W$, there are
\begin{itemize}
\item[$\bullet$] $r\in\N$,
\item[$\bullet$] $P_0,\ldots, P_r$ mutually disjoint projections with $P_1,\ldots,P_r$ of rank one and $P_0 = 1-P_1-\ldots-P_r$,
\item[$\bullet$] positive integers $k_1\leq\ldots\leq k_r$, and
\item[$\bullet$] analytic families of operators $E$ and $G$, which are defined in a neighborhood $\mathcal U$ of $z_0$ in $\Omega$ and which are invertible in $\mathcal U$,
\end{itemize}
such that for all $z\in\mathcal U$,
\begin{equation}
\label{eq:factorization}
W(z) = E(z) \left( P_0 + (z-z_0)^{k_1} P_1 + \ldots + (z-z_0)^{k_r} P_r \right) G(z)
\end{equation}
The sum $k_1+\ldots+k_r$ is called the \emph{algebraic multiplicity} of the eigenvalue $z_0$.

After these preparations we are ready to state our main result about eigenvalues of analytic families of operators in $\C\setminus[0,\infty)$.

\begin{theorem}\label{af}
Let $\rho\geq 0$ and $\sigma\in\R$ with $\rho+\sigma>0$ and let $p\geq 1$. Let $K(z)$, $z\in\C\setminus[0,\infty)$, be an analytic family of operators satisfying
$$
\|K(z)\|_p \leq M \delta(z)^{-\rho}|z|^{-\sigma}
\qquad\text{for all}\ z\in\C\setminus[0,\infty) \,.
$$
Let $z_j\in\C\setminus[0,\infty)$ be the eigenvalues of $1+K$ of finite type, repeated according to their algebraic multiplicity. Then, if $\rho>0$, for all $\epsilon,\epsilon'>0$,
$$
\sum_{|z_j|\leq M^{1/(\rho+\sigma)}} \delta(z_j)^{p\rho +1 + \epsilon} |z_j|^{- \frac12(p\rho + 1+\epsilon)+ \frac12(p\rho+2p\sigma-1+\epsilon)_+} \leq C M^{\frac1{2(\rho+\sigma)}\left(p\rho + 1+\epsilon+(p\rho+2p\sigma-1+\epsilon)_+ \right)}
$$
and
$$
\sum_{|z_j|\geq \nu M^{1/(\rho+\sigma)}} \delta(z_j)^{p\rho +1 + \epsilon} |z_j|^{\rho+\sigma - p\rho -1 -\epsilon - \epsilon'} \leq C' \nu^{-\epsilon'} M^{\frac{\rho+\sigma-\epsilon'}{\rho+\sigma}} 
\qquad\text{for all}\ \nu\geq 1 \,.
$$
Moreover, if $\rho=0$, then $|z_j|\leq M^{1/\sigma}$ for all $j$, and for all $\epsilon>0$,
$$
\sum \delta(z_j) |z_j|^{- \frac12+\frac12(2p\sigma-1+\epsilon)_+} \leq C M^{\frac1{2\sigma}\left(1+(2p\sigma-1+\epsilon)_+\right)} \,.
$$
Here $C$ (and $C'$) depend only on $\rho$, $\sigma$, $p$, $\epsilon$ (and $\epsilon'$).
\end{theorem}

The proof of Theorem \ref{af} relies on the identification of the eigenvalues of finite type of $1+K$, counting their algebraic multiplicities, with the zeroes of an analytic function, counting their orders. The analytic function in question is a regularized determinant and we refer to \cite{Si} for their definition and basic properties.

\begin{lemma}\label{zeroes}
Let $\Omega\subset\C$ be a connected, open set and $W(z)$, $z\in\Omega$, an analytic family of bounded operators such that for some $n\in\N$, $W(z)-1\in\mathfrak S^n$ for all $z\in\Omega$. Then $\det_n W$ is an analytic function in $\Omega$. Moreover, let $z_0\in\Omega$ and assume that $W(z_*)$ is invertible for some $z_*\in\Omega$. Then $\det_n W(z_0)=0$ if and only if $z_0$ is an eigenvalue of finite type of $W$. In this case the order of the zero of $\det_n W$ at $z_0$ coincides with the algebraic multiplicity of $z_0$ as an eigenvalue of $W$.
\end{lemma}

This lemma seems to be well-known and, for instance, in the proof of \cite[Thm.$\!$ 21]{LaSu} there are some hints on how to prove this. We have not found a complete proof in the literature, however, and so for the convenience of the reader we provide one in Appendix~\ref{sec:zeroes}.

Lemma \ref{zeroes} reduces the study of eigenvalues to the study of zeroes of analytic functions. A key ingredient in our proof is the following result of Borichev--Golinskii--Kupin \cite{BGK}, which generalizes Blaschke's theorem about the zeroes of bounded analytic functions to functions which blow up rapidly at the boundary.

\begin{proposition}\label{bgk}
Let $\alpha,\beta\geq 0$ and let $g$ be an analytic function in $\mathbb D$ satisfying $g(0)=1$ and
$$
\ln|g(w)|\leq K (1-|w|)^{-\alpha} |w+1|^{-\beta}
\qquad\text{for all}\ w\in\mathbb D \,.
$$
Let $w_j\in\mathbb D$ be the zeroes, counting multiplicities, of $g$. Then, for any $\epsilon>0$,
$$
\sum (1-|w_j|)^{\alpha+1+\epsilon} |w_j+1|^{(\beta-1+\epsilon)_+} \leq C_{\alpha,\beta,\epsilon} K
$$ 
If $\alpha=0$, then
$$
\sum (1-|w_j|) |w_j+1|^{(\beta-1+\epsilon)_+} \leq C_{1,\beta,\epsilon} K \,.
$$
\end{proposition}

\begin{proof}[Proof of Theorem \ref{af}]
After replacing $z$ by $M^{-1/(\rho+\sigma)}z$ we may assume that $M=1$.

\emph{Step 1.} We first assume that $\rho>0$ and show that there is a constant $A> 1$ (depending only on $p$) such that
\begin{equation}
\label{eq:afmain}
\sum \frac{\delta(z_j)^{p\rho+1+\epsilon} |z_j|^{\frac{(p\rho+2p\sigma-1+\epsilon)_+-(p\rho+1+\epsilon)}{2}}}{(|z_j|+a)^{(p\rho+1+\epsilon)+\frac{(p\rho+2p\sigma-1+\epsilon)_+}{2}}} \leq C a^{-(\rho+\sigma)-\frac{p\rho+1+\epsilon}{2}} 
\quad\text{for all}\ a \geq A \,.
\end{equation}
Here $\epsilon>0$ is arbitrary and $C$ depends on $\epsilon$, $p$, $\rho$ and $\sigma$. Let
$$
h(z) := \det {}_{\lceil p \rceil} (1+K(z)) \,,
$$
where $\lceil p \rceil$ is the smallest integer $\geq p$. According to Lemma \ref{zeroes} this is an analytic function of $z\in\C\setminus[0,\infty)$ and its zeroes, including multiplicities, coincide with the $z_j$. For $a>1$ we consider the conformal map $\psi:\mathbb D\to\C\setminus [0,\infty)$,
$$
\psi(w)= -a \left( \frac{1+w}{1-w} \right)^2
$$
and define the following function on the unit disk, $g:\mathbb D\to\C$,
$$
g(w) := \frac{h(\psi(w))}{h(-a)} \,.
$$
Clearly, this is an analytic function on $\mathbb D$ with $g(0)=1$ and, by the properties of $h$, its zeroes, counting multiplicities, coincide with the $\psi^{-1}(z_j)$. (Here we also use the fact that $h(-a)\neq 0$ since $\|K(-a)\|\leq \|K(-a)\|_p \leq a^{-\rho-\sigma}<1$.) Let us show that $g$ satisfies a bound as required for Proposition \ref{bgk}. In fact, we claim that
\begin{equation}
\label{eq:afproof}
\ln|g(w)| \leq C a^{-(\rho+\sigma)} (1-|w|)^{-p\rho} |1+w|^{-p\rho-2p\sigma}
\end{equation}
with a constant $C$ depending on $p$, $\rho$ and $\sigma$.

In fact, by inequalities about regularized determinants \cite[Thm. 9.2]{Si} we have, with constants $C$ depending only on $p$,
$$
\ln|h(z)| \leq C \|K(z)\|_p^p \leq C \delta(z)^{-p\rho} |z|^{-p\sigma}
$$
and
$$
\left||h(-a)|-1\right| \leq |h(-a)-1| \leq \|K(-a)\|_p \, e^{C(\|K(-a)\|_p^p+1)} \leq a^{-(\rho+\sigma)} \, e^{C(a^{-p(\rho+\sigma)}+1)} \,.
$$
(Strictly speaking, these inequalities are contained in \cite[Thm. 9.2]{Si} only for $p\in\N$, but the proof there works for any $p>0$.) The latter inequality implies that there are constants $C>0$ and $A> 1$, again depending only on $p$, such that
$$
\ln |h(-a)| \geq - C a^{-(\rho+\sigma)} 
\qquad\text{for all}\ a\geq A \,.
$$
This implies that
$$
\ln|g(w)|=\ln|h(\psi(w))| - \ln|h(-a)| \leq C \left(\delta(\psi(w))^{-p\rho} |\psi(w)|^{-p\sigma} + a^{-(\rho+\sigma)}\right).
$$
We now compute $\psi'(w) = -4a(1+w)/(1-w)^3$ and obtain, by Koebe's distortion theorem (see \cite[p. 9]{Po} and also \cite[Thm. 4.3.4]{DeHaKa}),
$$
\delta(\psi(w)) \geq \frac14 |\psi'(w)| (1-|w|) = \frac{a\,|1+w|\,(1-|w|)}{|1-w|^3}\,.
$$
This, together with the previous bound, yields
\begin{align*}
\ln|g(w)| & \leq C \left( a^{-p(\rho+\sigma)} (1-|w|)^{-p\rho} |1+w|^{-p\rho-2p\sigma} |1-w|^{3p\rho+2p\sigma} + a^{-(\rho+\sigma)}\right) \\
& \leq C \left( 8^{3p\rho+2p\sigma} a^{-p(\rho+\sigma)} (1-|w|)^{-p\rho} |1+w|^{-p\rho-2p\sigma} + a^{-(\rho+\sigma)} \right).
\end{align*}
It is easy to deduce \eqref{eq:afproof}. In fact, for the first term on the right side we use $a>1$, $p\geq 1$ and $\rho+\sigma\geq 0$ to bound $a^{-p(\rho+\sigma)}\leq a^{-(\rho+\sigma)}$. For the second term we use $p(\rho+\sigma)\geq 0$ and obtain $1 \leq 2^{p(\rho+\sigma)} |1+w|^{-p(\rho+\sigma)}$ and
$$
1 \leq 2^{p(\rho+\sigma)} |1+w|^{-p(\rho+\sigma)} \leq 2^{p(\rho+\sigma)} (1-|w|)^{-p\rho} |1+w|^{-p\sigma} \,,
$$
so $1 \leq 4^{p(\rho+\sigma)} (1-|w|)^{-p\rho} |1+w|^{-p\rho-2p\sigma}$. This proves \eqref{eq:afproof}.

Thus, we are in the situation of Proposition \ref{bgk} and we infer that for any $\epsilon>0$, with a constant depending only on $p$, $\rho$, $\sigma$ and $\epsilon$,
\begin{equation}
\label{eq:afmainproof}
\sum \left(1-\left|\psi^{-1}(z_j)\right|\right)^{p\rho+1+\epsilon} \left|1+ \psi^{-1}(z_j)\right|^{(p\rho+2p\sigma - 1 + \epsilon)_+} \leq C a^{-(\rho+\sigma)} \,.
\end{equation}
Since $\psi^{-1}(z) = (\sqrt{-z} - \sqrt{a})/(\sqrt{-z} + \sqrt{a})$, we have
$$
\left|1+\psi^{-1}(z)\right| = \frac{2\sqrt{|z|}}{\left|\sqrt{-z} + \sqrt{a}\right|} \geq 
\frac{\sqrt{2} \,\sqrt{|z|}}{\sqrt{|z| + a}} \,,
$$
and, again by Koebe's distortion theorem,
$$
1-\left|\psi^{-1}(z)\right| \geq \frac{\delta(z) \left|(\psi^{-1})'(z)\right|}{2} = \frac{\delta(z)\sqrt{a}}{2\sqrt{|z|}\left|\sqrt{-z}+\sqrt{a}\right|^2} 
\geq \frac{\delta(z)\sqrt{a}}{4\sqrt{|z|}(|z|+a)} \,,
$$
where we used $(\psi^{-1})'(z) = -\sqrt{a}/(\sqrt{-z}(\sqrt{-z}+\sqrt{a})^2)$. Inserting these bounds into \eqref{eq:afmainproof} we obtain \eqref{eq:afmain}.

\emph{Step 2.} Still assuming $\rho>0$ we now use \eqref{eq:afmain} to prove the bounds in the theorem.

First, in order to control the sum of terms with $|z_j|\leq 1$ we choose $a=A$ and bound $|z_j|+A \leq 1+A$. This yields the first bound in the theorem.

Next, in order to control the sum of terms with $|z_j|\geq \nu$ and $\nu\geq 1$ we multiply \eqref{eq:afmain} by $a^{(\rho+\sigma)+\frac{p\rho+1+\epsilon}{2}-1-\epsilon'}$ and integrate with respect to $a$ from $\nu A$ to $\infty$. (A similar idea appears already in \cite{DeHaKa0}.) The integral on the right side is finite and equals $(C/\epsilon')(\nu A)^{-\epsilon'}$. Therefore the integral of each term of the left side is finite as well (note that this implies $\rho+\sigma - (p\rho+1+\epsilon)/2-(p\rho+2p\sigma-1+\epsilon)_+/2 - \epsilon'<0$). We observe for $|z|\geq \nu$,
\begin{align*}
& \int_{\nu A}^\infty \frac{a^{(\rho+\sigma)+\frac{p\rho+1+\epsilon}{2}-1-\epsilon'}}{(|z|+a)^{(p\rho+1+\epsilon)+\frac{(p\rho+2p\sigma-1+\epsilon)_+}{2}}} \,da \\
& \quad = |z|^{(\rho+\sigma)-\frac{p\rho+1+\epsilon}{2}-\frac{(p\rho+2p\sigma-1+\epsilon)_+}{2}-\epsilon'} \int_{\nu A/|z|}^\infty \frac{r^{(\rho+\sigma)+\frac{p\rho+1+\epsilon}{2}-1-\epsilon'}}{(1+r)^{(p\rho+1+\epsilon)+\frac{(p\rho+2p\sigma-1+\epsilon)_+}{2}}} \, dr \\
& \quad \geq |z|^{(\rho+\sigma)-\frac{p\rho+1+\epsilon}{2}-\frac{(p\rho+2p\sigma-1+\epsilon)_+}{2}-\epsilon'} \int_{A}^\infty \frac{r^{(\rho+\sigma)+\frac{p\rho+1+\epsilon}{2}-1-\epsilon'}}{(1+r)^{(p\rho+1+\epsilon)+\frac{(p\rho+2p\sigma-1+\epsilon)_+}{2}}} \, dr \,.
\end{align*}
Inserting this into the integrated version of \eqref{eq:afmain} we obtain
$$
\sum_{|z_j|\geq \nu} \delta(z_j)^{p\rho+1+\epsilon} |z_j|^{(\rho+\sigma)-(p\rho+1+\epsilon)-\epsilon'} \leq C \nu^{-\epsilon'} \,,
$$
which is the second bound in the theorem.

\emph{Step 3.} Finally, we briefly comment on the case $\rho=0$. (The proof in this case is essentially contained in \cite{FrSa}.) The proof is similar, except we can apply Proposition \ref{bgk} with $\alpha=0$ and therefore the term $p\rho+1+\epsilon$ can be replaced by $1$. Thus the main inequality \eqref{eq:afmain} becomes
\begin{equation}
\label{eq:afmain0}
\sum \frac{\delta(z_j) |z_j|^{\frac{(2p\sigma-1+\epsilon)_+-1}{2}}}{(|z_j|+a)^{1+\frac{(2p\sigma-1+\epsilon)_+}{2}}} \leq C a^{-\sigma-\frac{1}{2}} \,,
\end{equation}
which immediately yields the bound in the theorem. Moreover, the fact that $|z_j|\leq 1$ follows from the fact that $1\leq \|K(z_j)\|\leq\|K(z_j)\|_p \leq |z_j|^{-\sigma}$.  This completes the proof of the theorem.
\end{proof}


\section{The Birman--Schwinger principle and the proof of Theorem \ref{main2}}\label{sec:bs}

In this section we prove Theorem \ref{main2} about the eigenvalues of Schr\"odinger operators by relating them to the eigenvalues of an analytic family of compact operators and by invoking Theorem \ref{af}. The link between Schr\"odinger eigenvalues and eigenvalues of compact operators is provided by the Birman--Schwinger principle. This principle is true in the non-selfadjoint case as well, and we begin by reviewing it. We present this in a more general set-up since it might be useful in other contexts as well.

Let $H_0$ be a self-adjoint, non-negative operator in a Hilbert space $\mathcal H$ and assume that $G_0$ and $G$ are operators from $\mathcal H$ to some Hilbert space $\mathcal G$ such that $\dom G_0\supset\dom H_0^{1/2}$, $\dom G\supset\dom H_0^{1/2}$ and
\begin{equation}
\label{eq:katoass}
G_0 (H_0 + 1)^{-1/2} \,,
\qquad
G (H_0 + 1)^{-1/2}
\qquad
\text{are compact}.
\end{equation}
As we shall recall in the appendix (Lemma \ref{defop}) under these assumptions the quadratic form
$$
\|H_0^{1/2} u\|^2 + (Gu,G_0u)
$$
with domain $\dom H_0^{1/2}$ is closed and sectorial and generates an $m$-sectorial operator $H$. (It is natural to use the notation
$$
H=H_0 + G^* G_0 \,,
$$
but we emphasize that we do not assume that $G$ is closable, so $G^*$ need not be densely defined. Not requiring $G$ to be closable is important in some applications, see the remark in Appendix \ref{sec:defop}.)

We recall in the appendix (Lemma \ref{spec}) that under assumption \eqref{eq:katoass} the spectrum of $H$ outside of $\sigma_\ess(H_0)$ is discrete and consists of eigenvalues of finite algebraic multiplicities. Our goal now is to characterize these eigenvalues and their multiplicities in terms of the Birman--Schwinger operator
\begin{equation}
\label{eq:bsdef}
K(z) := G_0 (H_0-z)^{-1} G^* \,,
\qquad z\in\rho(H_0) \,.
\end{equation}
Under assumption \eqref{eq:katoass} this is a compact operator in $\mathcal G$. (In order to avoid the problem mentioned above concerning the adjoint of $G$, we strictly speaking define $K(z)$ as
$$
K(z) = \left( G_0 (H_0 + 1)^{-1/2} \right) \left( (H_0+1)(H_0-z)^{-1} \right) \left( G (H_0 + 1)^{-1/2} \right)^* \,.
$$
Similar modifications are tacitly assumed in what follows.)

\begin{proposition}\label{bs}
Assume \eqref{eq:katoass} and let $z\in\rho(H_0)$. Then $z$ is an eigenvalue of $H$ iff it is an eigenvalue of finite type of the analytic family $1+K$. In this case the geometric (resp. algebraic) multiplicity of $z$ as an eigenvalue of $H$ coincides with the geometric (resp. algebraic) multiplicity of $z$ as an eigenvalue of $1+K$.
\end{proposition}

We emphasize that by `algebraic multiplicity of $z$ as an eigenvalue of $1+K$' we refer to the definition in Section \ref{sec:af}. This does \emph{not} necessarily coincide with the algebraic multiplicity of the eigenvalue $-1$ of $K(z)$.

A similar proposition is proved in \cite{LaSu} (see also \cite{GeHoNi}) under weaker assumptions on $H_0$ but stronger assumptions on $G$ and $G_0$. Their argument also works in our situation, but for the sake of completeness we provide details.

\begin{proof}
The equivalence and the equality of geometric multiplicities are standard results (see, e.g., \cite{GLMZ} and references therein). In fact, as is easily checked, the operators
$$
\dom H_0^{1/2}\ni\psi\mapsto G_0\psi \in\mathcal G \,,
\qquad
\mathcal G\ni \phi\mapsto -(H_0-z)^{-1} G^*\phi \in\dom H_0^{1/2}
$$
are well-defined, map $\ker(H_0-z)$ into $\ker(1+K(z))$ and $\ker(1+K(z))$ into $\ker(H_0-z)$, respectively, and are inverse to each other.

We now show the equality of the algebraic multiplicities. First, note that if $0$ is an eigenvalue of $1+K(z)$, then it is an eigenvalue of finite type. In fact, an operator that differs from the identity by a compact operator is Fredholm and we know from Lemma \ref{spec} that eigenvalues of $H$ outside of $\sigma_\ess(H_0)$ are isolated, so $1+K(\zeta)$ is invertible for $\zeta$ close to, but different from $z$. (Alternatively, one can deduce this from the analytic Fredholm alternative \cite[Cor.$\!$ 8.4]{GGK} applied to the family $1+K$, which is invertible in $\{\re \zeta<-M\}$ for some $M\geq 0$ since $H$ is $m$-sectorial.)

By Lemma \ref{spec} and the openness of $\rho(H_0)$ there is a closed disk $D$ centered at $z$ with $D\cap\sigma(H)=\{z\}$ and $D\cap\sigma(H_0)=\emptyset$. Let $\Gamma$ be the contour $\partial D$ oriented counter-clockwise. Then the algebraic multiplicity $m$ of the eigenvalue $z$ of $H$ satisfies
\begin{align*}
m = -\frac{1}{2\pi i} \tr \int_\Gamma (H-\zeta)^{-1} \,d\zeta = -\frac{1}{2\pi i} \int_\Gamma \tr \Xi[(H-\zeta)^{-1}] \,d\zeta \,,
\end{align*}
where $\Xi(W)(\zeta)$ denotes the principal part of a meromorphic operator family $W$ at $\zeta$. Since $\zeta\in\rho(H_0)\cap\rho(H)$ for $\zeta\in\Gamma$, one obtains from the resolvent identity \eqref{eq:resid}
\begin{align*}
\tr \Xi\left[(H-\zeta)^{-1}\right] & = \tr \Xi\left[(H-\zeta)^{-1}-(H_0-\zeta)^{-1} \right] \\
& = -\tr\Xi\left[ (H_0-\zeta)^{-1} G^* (1+K(\zeta))^{-1} G_0 (H_0-\zeta)^{-1} \right] \\
& = -\tr\Xi\left[ G_0 (H_0-\zeta)^{-2} G^* (1+K(\zeta))^{-1} \right] \\
& = - \tr\Xi\left[(1+K)'(1+K)^{-1}\right](\zeta) \,.
\end{align*}
In the next to last equality we used the commutativity lemma \cite[Lem.$\!$ XI.9.3]{GGK}. Thus, we obtain
\begin{align*}
m & = \frac{1}{2\pi i} \int_\Gamma \tr\Xi\left[(1+K)'(1+K)^{-1}\right](\zeta) \,d\zeta \\
& = \frac{1}{2\pi i} \tr \int_\Gamma (1+K(\zeta))'(1+K(\zeta))^{-1} \,d\zeta \,.
\end{align*}
According to the Gohberg--Rouch\'e theorem \cite[Thm.$\!$ XI.9.1]{GGK} the right side coincides with the algebraic multiplicity of the eigenvalue $z$ of $1+K$, as claimed.
\end{proof}

We now apply this approach to the Schr\"odinger operator $H=-\Delta+V$ by choosing
$$
H_0 = -\Delta \,,
\qquad
G = \sqrt{|V|} \,,
\qquad
G_0 = \sqrt{V}
$$
and $\mathcal H=\mathcal G=L^2(\R^d)$. Assumption \eqref{eq:katoass} follows from the following well-known lemma with the choices $W=\sqrt{|V|}$ and $W=\sqrt{V}$, whose proof we include for the sake of completeness.

\begin{lemma}\label{bssobolev}
Let $W\in L^{2\gamma+d}(\R^d)$, where $\gamma\geq 1/2$ if $d=1$, $\gamma>0$ if $d=2$ and $\gamma\geq 0$ if $d\geq 3$. Then for any $a>0$ the operator $W (-\Delta+a)^{-1/2}$ is compact and
\begin{equation}
\label{eq:bssobolev}
\left\| W(-\Delta+a)^{-1/2} \right\| \leq C_{\gamma,d}\, a^{-\gamma/(2\gamma+d)} \left( \int_{\R^d} |W|^{2\gamma+d} \,dx \right)^{1/(2\gamma+d)}
\end{equation}
with a constant $C_{\gamma,d}$ independent of $W$. If $\gamma=0$ and $d\geq 3$, this holds for $a=0$ as well.
\end{lemma}

\begin{proof}
Inequality \eqref{eq:bssobolev} with $a=1$ follows from Sobolev inequalities (see, e.g., \cite[Thms.$\!$ 8.3 and 8.5]{LiLo}). The case of arbitrary $a>0$ follows by scaling. The fact that $W(-\Delta+a)^{-1/2}$ is compact if $W$ is bounded and has support in a set of finite measure follows from Rellich's lemma (see, e.g., \cite[Thm.$\!$ 8.6]{LiLo}. The general case can be deduced from this, inequality \eqref{eq:bssobolev} and the fact that the limit of compact operators is compact. (Indeed, for given $\epsilon>0$ write $W=W_1+W_2$ with $W_1$ bounded and of support of finite measure and with $\|W_2\|_{2\gamma+d}\leq \epsilon$. Then, by \eqref{eq:bssobolev}, $W(-\Delta+a)^{-1/2}$ differs from the compact operator $W_1(-\Delta+a)^{-1/2}$ by an operator with norm bounded by $C_{\gamma,d} a^{-\gamma/(2\gamma+d)} \epsilon$.)
\end{proof}

The Birman--Schwinger operator in the Schr\"odinger case takes the form
\begin{equation}
\label{eq:bsso}
K(z) = \sqrt V (-\Delta-z)^{-1} \sqrt{|V|} \,,
\qquad z\in\C\setminus[0,\infty) \,.
\end{equation}

After these preliminaries we can give the

\begin{proof}[Proof of Theorem \ref{main2}]
According to Proposition \ref{res} the Birman--Schwinger operator \eqref{eq:bsso} satisfies
$$
\|K(z)\|_p \leq M \delta(z)^{-\rho} |z|^{-\sigma}
$$
with
$$
p=2(\gamma+d/2) \,,\quad \rho = \frac{\gamma-1/2}{\gamma+d/2} \,,\quad \sigma = \frac{1}{2(\gamma+d/2)} \,,\quad 
M = C \|V\|_{\gamma+d/2} \,. 
$$
We note that
$$
p\rho + 1 + \epsilon = 2\gamma+\epsilon \,,\quad
p\rho+2p\sigma-1+\epsilon = 2\gamma + \epsilon \,,\quad
\rho+\sigma = \frac{\gamma}{\gamma+d/2} \,.
$$
According to Proposition \ref{bs}, the eigenvalues of $-\Delta+V$ in $\C\setminus[0,\infty)$ coincide with the eigenvalues of $1+K$ of finite type in $\C\setminus[0,\infty)$, counting algebraic multiplicities. Thus, the claimed bounds follow from Theorem \ref{af} (with $\nu=\mu^{1/\gamma}$).
\end{proof}


\section{Proof of Theorem \ref{main3}}

The proof of Theorem \ref{main3} is of a different nature than the proofs of Theorem \ref{main1} and \ref{main2} and, in particular, does not rely on the resolvent bounds of Proposition \ref{res}. It rather uses a matrix inequality due to Hansmann \cite{Ha}, which extends a classical result of Kato to non-self-adjoint perturbations, together with a standard resolvent bound coming from the Kato--Seiler--Simon inequality.

\begin{proof}[Proof of Theorem \ref{main3}]
\emph{Step 1.} We shall show that there are constants $C_{\gamma,d}$ and $L_{\gamma,d}$ such that
\begin{align}
\label{eq:main3proofkey}
\sum_j \frac{\delta(E_j)^{\gamma+d/2}}{(|E_j|+a)^{2\gamma+d}} \leq & L_{\gamma,d} \ a^{-2\gamma-d/2} \int_{\R^d} |V|^{\gamma+d/2} \,dx \notag \\
& \text{for all}\ a\geq C_{\gamma,d} \left( \int_{\R^d} |V|^{\gamma+d/2}\,dx \right)^{1/\gamma} \,.
\end{align}
A simple computation \cite[(3.41)]{Ha0}) shows that for all $a>0$ and $E\in\C$,
$$
\dist( (E+a)^{-1},[0,a^{-1}]) \geq \frac{\delta(E)}{8|E+a| (|E|+a)} \,.
$$
Since $(E_j+a)^{-1}$ are the eigenvalues, counting algebraic multiplicities, of $(-\Delta+V+a)^{-1}$ and since $[0,a^{-1}]$ is the spectrum of $(-\Delta+a)^{-1}$, we can use Hansmann's bound \cite{Ha} to conclude that
\begin{align}
\label{eq:ha}
\frac 1{8^{\gamma+d/2}} \sum \frac{\delta(E_j)^{\gamma+d/2}}{(|E_j| + a)^{2\gamma+d}} 
& \leq \sum \dist( (E+a)^{-1},[0,a^{-1}])^{\gamma+d/2} \notag \\
& \leq \left\|(-\Delta+V+a)^{-1} - (-\Delta +a)^{-1} \right\|_{\gamma+d/2}^{\gamma+d/2} \,.
\end{align}
Thus, we need an upper bound on the Schatten norm on the right side. According to the resolvent identity \eqref{eq:resid},
\begin{align*}
& (-\Delta+V+a)^{-1} - (-\Delta +a)^{-1} \\
& \qquad = -(\Delta+a)^{-1} \sqrt{|V|} \left( 1 + \sqrt{V} (-\Delta+a)^{-1} \sqrt{|V|} \right)^{-1} \sqrt{V} (-\Delta+a)^{-1} \,.
\end{align*}
From \eqref{eq:bssobolev} we obtain
$$
\left\| \sqrt{V}(-\Delta+a)^{-1} \sqrt{|V|} \right\| \leq C_{\gamma,d}' \, a^{-\gamma/(\gamma+d/2)} \left( \int_{\R^d} |V|^{\gamma+d/2} \,dx \right)^{1/(\gamma+d/2)} \,.
$$
Thus, if $a\geq C_{\gamma,d} \left( \int |V|^{\gamma+d/2}\,dx\right)^{1/\gamma}$ with $C_{\gamma,d} := (2C'_{\gamma,d})^{(\gamma+d/2)/\gamma}$, then $\|\sqrt{V} (-\Delta+a)^{-1} \sqrt{|V|}\|\leq 1/2$ and therefore
$$
\left\| \left( 1 + \sqrt{V} (-\Delta+a)^{-1} \sqrt{|V|} \right)^{-1} \right\| \leq 2 \,.
$$
Thus, by the H\"older inequality for trace ideals and the Kato--Seiler--Simon inequality \cite[Thm. 4.1]{Si},
\begin{align*}
& \left\|(-\Delta+V+a)^{-1} - (-\Delta +a)^{-1} \right\|_{\gamma+d/2}^{\gamma+d/2} \leq 2^{\gamma+d/2} \left\|(-\Delta+a)^{-1} \sqrt{|V|}\right\|_{2(\gamma+d/2)}^{2(\gamma+d/2)} \\
& \qquad \leq 2^{\gamma+d/2} (2\pi)^{-d} \int_{\R^d} \frac{dp}{(p^2+a)^{2(\gamma+d/2)}} \int_{\R^d} |V|^{\gamma+d/2}\,dx \\
& \qquad \leq 2^{\gamma+d/2} (2\pi)^{-d} a^{-2\gamma-d/2} \int_{\R^d} \frac{dp}{(p^2+1)^{2(\gamma+d/2)}} \int_{\R^d} |V|^{\gamma+d/2}\,dx \,.
\end{align*}
Inserting this into \eqref{eq:ha} we obtain \eqref{eq:main3proofkey}.

\emph{Step 2.} We now deduce the theorem from \eqref{eq:main3proofkey} by similar arguments as \cite{DeHaKa0} and in the proof of Proposition \ref{af}. We abbreviate $A:=C_{\gamma,d} \left(\int |V|^{\gamma+d/2}\right)^{1/\gamma}$. First, in order to deal with eigenvalues satisfying $|E_j|\leq A$ we choose $a=A$ in \eqref{eq:main3proofkey} and bound $|E_j|+a\leq 2A$. This yields
\begin{align*}
\sum_{|E_j|^\gamma \leq C_{\gamma,d}^\gamma \int |V|^{\gamma+d/2}\,dx} \delta(E_j)^{\gamma+d/2} & \leq 2^{2\gamma+d} L_{\gamma,d} \, A^{d/2} \int_{\R^d} |V|^{\gamma+d/2}\,dx \\
& = 2^{2\gamma+d} L_{\gamma,d} C_{\gamma,d}^{d/2} \left(\int_{\R^d} |V|^{\gamma+d/2}\,dx \right)^{(\gamma+d/2)/\gamma} \,,
\end{align*}
which is the first inequality in the theorem.

Next, in order to deal with eigenvalues satisfying $|E_j|\geq \mu^{1/\gamma} A$ for some $\mu\geq 1$ we multiply \eqref{eq:main3proofkey} by $a^{2\gamma+d/2-1-\epsilon'}$ and integrate from $\mu^{1/\gamma} A$ to $\infty$. The integral on the right side is finite and equals
$$
L_{\gamma,d} (\epsilon')^{-1} \mu^{-\epsilon'/\gamma} A^{-\epsilon'} \int_{\R^d} |V|^{\gamma+d/2}\,dx \,.
$$
Therefore the integral on the left side is finite as well. We observe that
\begin{align*}
\int_{\mu^{1/\gamma} A}^\infty \frac{a^{2\gamma+d/2-1-\epsilon'}}{(|E|+a)^{2\gamma+d}} \,da = 
|E|^{-d/2-\epsilon'} \int_{\mu^{1/\gamma} A/|E|}^\infty \frac{r^{2\gamma+d/2-1-\epsilon'}}{(1+r)^{2\gamma+d}} \,dr \,.
\end{align*}
Thus, for eigenvalues satisfying $|E_j|\geq\mu^{1/\gamma} A$ we have
$$
\int_{\mu^{1/\gamma} A}^\infty \frac{a^{2\gamma+d/2-1-\epsilon'}}{(|E_j|+a)^{2\gamma+d}} \,da \geq 
|E_j|^{-d/2-\epsilon'} \int_1^\infty \frac{r^{2\gamma+d/2-1-\epsilon'}}{(1+r)^{2\gamma+d}} \,dr =: |E_j|^{-d/2-\epsilon'} c_{\gamma,d,\epsilon'} \,.
$$
Inserting this into the integrated version of \eqref{eq:main3proofkey} we obtain
\begin{align*}
\sum_{|E_j|^\gamma \geq \mu C_{\gamma,d}^\gamma \int |V|^{\gamma+d/2}\,dx} \delta(E_j)^{\gamma+d/2} |E_j|^{-d/2-\epsilon'} 
& \leq \frac{L_{\gamma,d}}{\epsilon' \, c_{\gamma,d,\epsilon'}} \mu^{-\frac{\epsilon'}{\gamma}} A^{-\epsilon'} \int_{\R^d} |V|^{\gamma+d/2}\,dx \\
& = \frac{L_{\gamma,d}\, C_{\gamma,d}^{-\epsilon'}}{\epsilon' \, c_{\gamma,d,\epsilon'}} \mu^{-\frac{\epsilon'}{\gamma}} \left( \int_{\R^d} |V|^{\gamma+d/2}\,dx \right)^\frac{\gamma-\epsilon'}{\gamma},
\end{align*}
which is the claimed inequality.
\end{proof}


\appendix

\section{Comparison of our results with \cite{DeHaKa}}\label{sec:dhk}

For the convenience of the reader we formulate \cite[Thm.$\!$ 7.2.3]{DeHaKa} in our notation in order to facilitate the comparison with our Theorems \ref{main2} and \ref{main3}. This theorem depends on a number $\omega\leq 0$ with $\sigma(-\Delta+V)\subset\{ E\in\C: \re E\geq \omega\}$. If we look for bounds that only depend on the $L^{\gamma+d/2}$ norm of $V$, we are forced to choose
$$
\omega= - \left( C_{\gamma,d} \int |V|^{\gamma+d/2}\,dx \right)^{1/\gamma}
$$
with some constant $C_{\gamma,d}$. (Indeed, even for real $V$ this lower bound on the spectrum of $-\Delta+V$ can be attained.) With this choice of $\omega$, \cite[Thm.$\!$ 7.2.3]{DeHaKa} can be equivalently stated as follows.

\begin{theorem}[\cite{DeHaKa}]\label{dhk}
Let $\gamma\geq 2-d/2$ if $d\leq 3$ and $\gamma>0$ if $d\geq 4$. Then the eigenvalues $E_j\in\C\setminus[0,\infty)$ of $-\Delta+V$ in $L^2(\R^d)$, repeated according to their algebraic multiplicity, satisfy:
\begin{enumerate}
\item If $\gamma< d/2$, for any $\epsilon>0$,
$$
\left( \sum_{|E_j|^\gamma\leq C_{\gamma,d} \int |V|^{\gamma+d/2}\,dx} \frac{\delta(E_j)^{\gamma+d/2+\epsilon}}{|E_j|^{(\gamma+d/2+\epsilon)/2}} \right)^{\frac{2\gamma}{\gamma+d/2+\epsilon}} \leq L_{\gamma,d,\epsilon} \int_{\R^d} |V|^{\gamma+d/2}\,dx \,.
$$
\item If $\gamma\geq d/2$, for any $\epsilon>0$,
$$
\left( \sum_{|E_j|^\gamma\leq C_{\gamma,d} \int |V|^{\gamma+d/2}\,dx} \frac{\delta(E_j)^{\gamma+d/2+\epsilon}}{|E_j|^{d/2}} \right)^{\frac{\gamma}{\gamma+\epsilon}} \leq L_{\gamma,d,\epsilon} \int_{\R^d} |V|^{\gamma+d/2}\,dx \,.
$$
\item For any $0<\epsilon<\gamma$,
$$
\left( \sum_{|E_j|^\gamma\geq C_{\gamma,d} \int |V|^{\gamma+d/2}\,dx} \frac{\delta(E_j)^{\gamma+d/2+\epsilon}}{|E_j|^{d/2+2\epsilon}} \right)^{\frac{\gamma}{\gamma-\epsilon}} \leq L_{\gamma,d,\epsilon} \int_{\R^d} |V|^{\gamma+d/2}\,dx \,.
$$
\end{enumerate}
\end{theorem}

It is not completely obvious that this theorem is equivalent to \cite[Thm.$\!$ 7.2.3]{DeHaKa} with the above choice of $\omega$ and we will justified this below. Accepting this claim for the moment, let us compare Theorem \ref{dhk} with our main results, Theorems \ref{main2} and \ref{main3}.

Theorem \ref{main3} implies (3). In fact,
$$
\frac{\delta(E_j)^{\gamma+d/2+\epsilon}}{|E_j|^{d/2+2\epsilon}}
= \frac{\delta(E_j)^{\gamma+d/2}}{|E_j|^{d/2+\epsilon}} \left( \frac{\delta(E_j)}{|E_j|} \right)^\epsilon
$$
and the last factor is $\leq 1$.\\
Theorem \ref{main2} implies (1) for $\gamma\leq d/6$. In fact,
$$
\frac{\delta(E_j)^{\gamma+d/2+\epsilon}}{|E_j|^{(\gamma+d/2+\epsilon)/2}}
\leq \delta(E_j)^{(\gamma+d/2+\epsilon)/2} = \left( \frac{\delta(E_j)}{|E_j|} \right)^{(\gamma+d/2+\epsilon)/2} \left(\frac{|E_j|}\omega\right)^{(\gamma+d/2+\epsilon)/2} \omega^{(\gamma+d/2+\epsilon)/2} \,.
$$
If $\gamma\leq d/6$ we can find an $\epsilon'>0$ such that $2\gamma+\epsilon'\leq (\gamma+d/2+\epsilon)/2$. Thus the first factor on the right side can be bounded by the corresponding factor with exponent $2\gamma+\epsilon'$. Moreover, the second factor is $\leq 1$. Therefore we obtain (1) from Theorem \ref{main2}.\\
The bound on the accumulation rate for $E_j\to E$ with $E\in (0,\infty)$ given in Theorem \ref{main2} (resp. Theorem \ref{main3}) is stronger than that in (1) (resp. (2)). This simply follows since the factors of $|E_j|$ in the denominators in (1) and (2) are irrelevant in this case.\\
For eigenvalues with $E_j\to 0$, each of Theorem \ref{main2} and (1) can be stronger than the other and each of Theorem \ref{main3} and (2) can be stronger than the other. Unless $\delta(E_j)/|E_j|$ is very small, however, (1) and (2) tend to give better results. In particular, they almost coincide with the bound from \cite{FrLaLiSe} on the accumulation rate of eigenvalues near zero, but outside of sectors containing $[0,\infty)$. Note, however, that neither Theorem \ref{dhk} nor the bounds from \cite{FrLaLiSe} are applicable for $1/2\leq\gamma<1$ if $d=1$, for $0<\gamma<1$ if $d=2$ and for $0<\gamma<1/2$ if $d=3$, whereas our Theorems \ref{main2} and \ref{main3} provide bounds in these cases.

\medskip

We are now going to justify the claim that Theorem \ref{dhk} is equivalent to Theorem \cite[Thm.$\!$ 7.2.3]{DeHaKa} with the above choice of $\omega$.

\emph{Equivalence of $\mathrm{(3)}$.} In the sums corresponding to $|E_j|\geq\omega$, the term $|E_j|+\omega$ in \cite[Thm.$\!$ 7.2.3]{DeHaKa} is equivalent to $|E_j|$. Therefore, we see that the bounds in \cite[(7.2.7)]{DeHaKa} and \cite[(7.2.8)]{DeHaKa} for eigenvalues $|E_j|\geq\omega$ coincide with each other and with our (3). We finally note that
$$
\frac{\delta(E_j)^{\gamma+d/2+\epsilon}}{|E_j|^{d/2+2\epsilon}}\, \omega^\epsilon
= \frac{\delta(E_j)^{\gamma+d/2}}{|E_j|^{d/2}} \left( \frac{\delta(E_j)}{|E_j|} \right)^\epsilon \left( \frac{\omega}{|E_j|} \right)^\epsilon \,,
$$
where the last two factors on the right side are $\leq 1$. This implies that the bound in (3) becomes stronger the smaller $\epsilon$ is, and so the condition $\epsilon<1$ in \cite[Thm.$\!$ 7.2.3]{DeHaKa} can be replaced by our condition $\epsilon<\gamma$.

\emph{Equivalence of $\mathrm{(2)}$.} Note that if $\gamma\geq d/2$ the bound \cite[(7.2.8)]{DeHaKa} is not applicable, so in this case we concentrate on \cite[(7.2.7)]{DeHaKa}. In the sum corresponding to $|E_j|\leq\omega$, the term $|E_j|+\omega$ is equivalent to $\omega$. In this way we immediately obtain our (2). We finally note that
$$
\frac{\delta(E_j)^{\gamma+d/2+\epsilon}}{|E_j|^{d/2}}\, \omega^{-\epsilon}
= \frac{\delta(E_j)^{\gamma+d/2}}{|E_j|^{d/2}}
\left( \frac{\delta(E_j)}{|E_j|} \right)^\epsilon
\left( \frac{|E_j|}\omega \right)^{\epsilon} \,,
$$
and again the last two factors on the right side are $\leq 1$. This implies that the bound in (2) becomes stronger the smaller $\epsilon$ is, and so the condition $\epsilon<1$ in \cite[Thm.$\!$ 7.2.3]{DeHaKa} can be dropped.

\emph{Equivalence of $\mathrm{(1)}$.} We have to consider both \cite[(7.2.7) and (7.2.8)]{DeHaKa}. As in the previous case we replace $|E_j|+\omega$ by $\omega$ if $|E_j|\leq\omega$. The bound \cite[(7.2.8)]{DeHaKa} then becomes our (1). We write
$$
\frac{\delta(E_j)^{\gamma+d/2+\epsilon}}{|E_j|^{(\gamma+d/2+\epsilon)/2}}\, \omega^{-\epsilon/2}
= \frac{\delta(E_j)^{\gamma+d/2}}{|E_j|^{(\gamma+d/2)/2}}
\left( \frac{\delta(E_j)}{|E_j|} \right)^\epsilon
\left( \frac{|E_j|}\omega \right)^{\epsilon/2} \,,
$$
and again the last two factors on the right side are $\leq 1$. This implies that the bound in (1) becomes stronger the smaller $\epsilon$ is, and so the conditions $\epsilon<1$ and $\epsilon<d/2-\gamma$ in \cite[Thm.$\!$ 7.2.3]{DeHaKa} can dropped. On the other hand, \cite[(7.2.7)]{DeHaKa} gives a bound of the same form as in (2), but with the condition that $\epsilon\geq d/2-\gamma$. As we observed there, we would like to choose $\epsilon$ as small as possible and we arrive at
$$
\left( \sum_{|E_j|^\gamma\leq C_{\gamma,d} \int |V|^{\gamma+d/2}\,dx} \frac{\delta(E_j)^d}{|E_j|^{d/2}} \right)^{\frac{\gamma}{d/2}} \leq L_{\gamma,d} \int_{\R^d} |V|^{\gamma+d/2}\,dx \,.
$$
But this is just the bound in (1) for $\epsilon=d/2-\gamma>0$, so this yields nothing new.

This completes our discussion of Theorem \ref{dhk}.

\medskip

\emph{Remark.} For the sake of completeness we also mention the bound
\begin{equation}
\label{eq:dhkaverage}
\sum_j \frac{\delta(E_j)^{\gamma+d/2+\epsilon}}{|E_j|^{d/2+\epsilon}} \leq L_{\gamma,d,\epsilon} \int_{\R^d} |V|^{\gamma+d/2} \,dx
\qquad\text{if} \ \gamma\geq 1
\ \text{and}\ \epsilon>0
\end{equation}
from \cite{DeHaKa0} which is obtained by averaging a bound from \cite{FrLaLiSe} over a free parameter. For a sequence of eigenvalues $(E_j)$ with $E_j\to E\in (0,\infty)$, the bounds of Theorems \ref{main2} and \ref{main3} are stronger are stronger than \eqref{eq:dhkaverage}.


\section{Relatively form-compact perturbations of non-negative operators}\label{sec:defop}

In this section we provide the details for the definition of the operator $H$ in Section \ref{sec:bs}. Recall that the Birman--Schwinger operator $K(z)$ was defined in \eqref{eq:bsdef}.

\begin{lemma}\label{defop}
Under assumption \eqref{eq:katoass}, the quadratic form $\|H_0^{1/2} u\|^2 + (Gu,G_0u)$ with domain $\dom H_0^{1/2}$ is closed and sectorial and generates an $m$-sectorial operator $H$. Let $z\in\rho(H_0)$. Then the operator $1+K(z)$ is boundedly invertible if and only if $z\in\rho(H)$, and in this case
\begin{equation}
\label{eq:resid}
(H-z)^{-1} = (H_0-z)^{-1} - (H_0-z)^{-1} G^* \left( 1+ K(z)\right)^{-1} G_0 (H_0-z)^{-1} \,.
\end{equation}
\end{lemma}

\begin{proof}
By assumption \eqref{eq:katoass} we can write, for any $\epsilon>0$,
$$
G (H_0+1)^{-1/2} = F + R \,,
$$
with $F$ finite rank and $\|R\|\leq \epsilon$. Moreover, since $F$ is finite rank and since $\dom H_0^{1/2}$ is dense, we may assume that $F(H_0+1)^{1/2}$ extends to a bounded operator. (Indeed, if $F=\sum |f_j\rangle\langle g_j|$ we can modify the $g_j$ such that $g_j\in\dom H_0^{1/2}$ and absorb the error due to this modification into $R$.) This proves that for $u\in\dom H_0^{1/2}$,
$$
\|G u\| \leq \|R(H_0+1)^{1/2} u\| + \|F(H_0+1)^{1/2} u\| \leq \epsilon \|(H_0+1)^{1/2} u\| + \|F(H_0+1)^{1/2}\| \|u\| \,.
$$
This and a similar bound for $G_0$ implies that for any $\epsilon>0$ there is a $C_\epsilon$ such that for all $u\in\dom H_0^{1/2}$,
$$
\left| (G u,G_0 u) \right| \leq \epsilon \|H_0^{1/2} u\|^2 + C_\epsilon \|u\|^2 \,.
$$
From \cite[Thm.$\!$ VI.3.4]{Ka} we conclude that the quadratic form $h[u]=\|H_0^{1/2} u\|^2 + (Gu,G_0u)$ with domain $\dom H_0^{1/2}$ is closed and sectorial and generates an $m$-sectorial operator $H$. Moreover, the usual resolvent identities
\begin{equation}
\label{eq:residstandard}
(H-z)^{-1} - (H_0-z)^{-1} = - (H-z)^{-1} G^* G_0 (H_0-z)^{-1}
\end{equation}
and
\begin{equation}
\label{eq:residstandard2}
(H-z)^{-1} - (H_0-z)^{-1} = - (H_0-z)^{-1} G^* G_0 (H_0-z)^{-1}
\end{equation}
are valid for $z\in\rho(H_0)\cap\rho(H)$. (We note that, since the quadratic form $h$ has domain $\dom H_0^{1/2}$, the operator $G(H+M)^{-1/2}$ is bounded for sufficiently large $M$ and therefore $(H-z)^{-1} G^*$ can be defined as $(H-z)^{-1} (H+M)^{1/2} (G(H+M)^{-1/2})^*$.)

It follows from \eqref{eq:residstandard} and \eqref{eq:residstandard2} that
$$
\left( 1- G_0 (H-z)^{-1} G^* \right)\left( 1 + K(z) \right) = 1
$$
and
$$
\left( 1 + K(z) \right) \left( 1- G_0 (H-z)^{-1} G^* \right) = 1 \,,
$$
which means that for $z\in\rho(H_0)\cap\rho(H)$ the operator $1+K(z)$ is invertible. Since $K(z)$ is compact, the inverse is bounded.

Now assume conversely that $z\in\rho(H_0)$ and that $1+K(z)$ is boundedly invertible. We note that for $v\in\dom H_0^{1/2}$, $f\in\mathcal H$ and $F\in\mathcal G$ one has
$$
h[v,(H_0-z)^{-1}f] - z(f,(H_0-z)^{-1} f) = (v,f) + (G v, G_0 (H_0-z)^{-1}f)
$$
and
$$
h[v,(H_0-z)^{-1} G^*F] - z(v,(H_0-z)^{-1} G^*F) = (Gv,F) + (Gv,K(z)F) \,.
$$
Using these identities one easily verifies that for $v\in\dom H_0^{1/2}$ and $f\in\mathcal H$ one has
\begin{align*}
& h[v,(H_0-z)^{-1}f - (H_0-z)^{-1} G^* \left( 1+ K(z)\right)^{-1} G_0 (H_0-z)^{-1}f] \\
& \qquad - z(v, (H_0-z)^{-1}f - (H_0-z)^{-1} G^* \left( 1+ K(z)\right)^{-1} G_0 (H_0-z)^{-1}f ) = (v,f) \,.
\end{align*}
This proves that $z\in\rho(H)$ and that $(H-z)^{-1}$ is given by \eqref{eq:resid}.
\end{proof}

For a closed operator $T$ in a Hilbert space with define the set
\begin{align*}
\rho_\ess (T):= & \{ z\in\C:\ \ran(T-z) \ \text{is closed and at least one of} \\
& \qquad\qquad \dim\ker(T-z)\ \text{and}\ \codim\ran (T-z)\ \text{is finite} \}
\end{align*}
and, following Kato, call
$$
\sigma_\ess(T) := \C\setminus\rho_\ess(T)
$$
the \emph{essential spectrum} of $T$. (There are several different notions of essential spectrum \cite{EdEv}, but we will see below that they all coincide in our situation.) The definition of $\rho_\ess(T)$ should be compared with that of the resolvent set
$$
\rho (T):= \{ z\in\C:\ \ran(T-z) \ \text{is closed and}\ \dim\ker(T-z) = \codim\ran (T-z) = 0 \} \,.
$$
According to \cite[Thm. IV.5.17]{Ka}, $\rho_\ess(T)$ is open and $\sigma_\ess(T)$ is closed.

We now return to our situation where $T$ is a perturbation of a self-adjoint, non-negative operator and the perturbation satisfies a relative form-compactness condition.

\begin{proposition}\label{spec}
Under assumption \eqref{eq:katoass} one has
\begin{equation}
\label{eq:weyl}
\sigma_\ess(H)=\sigma_\ess(H_0)
\end{equation}
and
\begin{align}
\label{eq:discev}
& \sigma(H)\setminus\sigma_\ess(H) \notag \\
& = \{ z\in\C:\ \ran(T-z) \ \text{is closed and}\ 0< \dim\ker(T-z) = \codim\ran (T-z) <\infty \} \,.
\end{align}
The set \eqref{eq:discev} is at most countable and consists of eigenvalues of finite algebraic multiplicities which are isolated in $\sigma(H)$.
\end{proposition}

\begin{proof}
We follow \cite[Sec.$\!$ IV.5.6]{Ka}. Since $H$ is $m$-sectorial, its resolvent set is non-empty and therefore \eqref{eq:katoass} and \eqref{eq:resid} imply that $(H-z)^{-1}-(H_0-z)^{-1}$ is compact for $z\in\rho(H_0)\cap\rho(H)$. According to \cite[Prob.$\!$ IV.5.38]{Ka}, this implies \eqref{eq:weyl}.

Note that for $z\in\rho_\ess(H)$ the index $\ind(H-z) = \dim\ker(H-z) - \codim \ran(H-z)$ is well-defined (although a-priori it might be equal to $+\infty$ or $-\infty$). Clearly, $\ind(H-z)=0$ for $z\in\rho(H)\subset\rho_\ess(H)$. Since $H_0$ is non-negative, $\rho_\ess(H)=\rho_\ess(H_0)$ is connected and by \cite[Thm.$\!$ IV.5.17]{Ka} the index is constant on connected components of $\rho_\ess(H)$. Therefore we conclude that $\ind(H-z)=0$ for all $z\in\rho_\ess(H)$. This proves \eqref{eq:discev}.

According to \cite[Thm.$\!$ IV.5.31]{Ka} for each $z_0\in\rho_\ess(H)$ there is an $r>0$ such that $\dim\ker(H-z)$ and $\codim\ran(H-z)$ are constant for all $0<|z-z_0|<r$. This fact implies that the set \eqref{eq:discev} is at most countable and has no accumulation points in $\C\setminus\sigma_\ess(H)$. Finally, the fact that points in \eqref{eq:discev} have finite algebraic multiplicities follows from \cite[Thm.$\!$ IV.5.28]{Ka} (taking also into account \cite[Thm.$\!$ IV.5.10]{Ka}). This completes the proof.
\end{proof}

\emph{Remark.}
The fact that $G_0$ and $G$ need not be closable is crucial when our construction is applied to define the Laplacian on a open set $\Omega\subset\R^d$ with Robin boundary conditions. Let $\sigma$ be the (complex-valued) function on $\partial\Omega$ which appears in the Robin boundary condition. Then the corresponding operator can be defined by taking $H_0$ the Neumann Laplacian in $\mathcal H=L^2(\Omega)$ and by taking $G_0$ (resp. $G$) the operator from $L^2(\Omega)$ to $\mathcal G=L^2(\partial\Omega)$ given by the trace operator followed by multiplication by $\sqrt\sigma$ (resp. $\sqrt{|\sigma|}$). These operators are not closable but, by Sobolev trace theorems, \eqref{eq:katoass} is satisfied provided $\partial\Omega$ is sufficiently regular and $\sigma\in L^p$ for a suitable $p$. A similar remark applies to Schr\"odinger operators with potentials supported on hypersurfaces.


\section{Zeroes of regularized determinants}\label{sec:zeroes}

\begin{proof}[Proof of Lemma \ref{zeroes}]
A proof of the fact that $\det_n W$ is analytic can be found, for instance, in \cite{Si1}. 

Clearly, if $\ker W(z_0)\neq\{0\}$, then $\det_n W(z_0)=0$. Now assume conversely that $\det_n W(z_0)=0$ and that $W(z_*)$ is invertible for some $z_*\in\Omega$. Then $\det_n W(z_0)=0$ implies that $\ker W(z_0)\neq\{0\}$ and, since $W(z_0)-1$ is compact, $W(z_0)$ is Fredholm. Moreover, since the analytic function $z\mapsto \det_n W(z)$ does not vanish at $z=z_*$ and therefore is not identically zero, there is an $\epsilon>0$ such that $\det_n W(z)\neq 0$ for all $0<|z-z_0|<\epsilon$. Thus, $\ker W(z)=\{0\}$ for all $0<|z-z_0|<\epsilon$ and, since $W(z)-1$ is compact, this implies that $W(z)$ is invertible for $0<|z-z_0|<\epsilon$. This means that $z_0$ is an eigenvalue of finite type of $W$.

In order to show that the algebraic multiplicity coincides with the order of zero of the determinant we use the factorization \eqref{eq:factorization}. (In the proof of \cite[Thm.$\!$ 21]{LaSu} it is mentioned that Lemma \ref{zeroes} follows from \eqref{eq:factorization}, but, since this is not completely obvious, we provide some details.) Clearly, $P_0 + (z-z_0)^{k_1} P_1 + \ldots + (z-z_0)^{k_r} P_r$ differs from the identity by a finite rank operator. Moreover, from the proof of \cite[Thm.$\!$ XI.8.1]{GGK} we know that $G(z)$ differs from the identity by a finite rank operator and that $E(z)$ differs from $W(z)$ be a finite rank operator. Therefore, under our assumptions, $E(z)$ differs from the identity by an operator in $\mathfrak S^n$. Lemma \ref{detprod} below shows that there is an analytic function $f$ in a neighborhood of $z_0$ such that
$$
\textstyle{\det_n}W(z) = \det\left( P_0 + (z-z_0)^{k_1} P_1 + \ldots + (z-z_0)^{k_r} P_r \right) \, \textstyle{\det_n}\left( E(z) G(z) \right) \, e^{f(z)} \,.
$$
Since $E(z)$ and $G(z)$ are invertible near $z_0$, we conclude that the order of zero of $\det_n W$ at $z_0$ coincides with the corresponding order of the first factor on the right side. But, by an explicit computation,
$$
\det \left( P_0 + (z-z_0)^{k_1} P_1 + \ldots + (z-z_0)^{k_r} P_r \right) = (z-z_0)^{k_1+\ldots+k_r} \,,
$$
which proves the lemma.
\end{proof}

The following result was used in the proof of Lemma \ref{zeroes}. Its proof follows closely that of \cite[Lem.$\!$ 1.5.10]{Ha0}.

\begin{lemma}\label{detprod}
For each $n\in\N$ there is a non-commutative polynomial $p_n(\cdot,\cdot,\cdot)$ such that for all $K,L\in\mathfrak S^n$ and finite rank $F$, $p_n(K,F,L)$ is finite rank and
$$
\textstyle{\det_n}\left( (1+K)(1+F)(1+L) \right) = \det(1+F)\, \textstyle{\det_n}\left( (1+K)(1+L) \right) e^{\tr p_n(K,F,L)} \,.
$$
\end{lemma}

\begin{proof}
The formula is clearly true for $n=1$ with $P_1\equiv 0$, so we may assume that $n\geq 2$. Let us first prove the lemma in the case where $K$ and $L$ are finite rank. In this case, we have
\begin{align*}
& \textstyle{\det_n}\left( (1+K)(1+F)(1+L) \right)
= \det\left( (1+K)(1+F)(1+L) \right) \\
& \qquad\qquad \times \exp\left( \sum_{m=1}^{n-1} \frac{(-1)^m}{m} \tr \left(K+F+L+KF+KL+FL+KFL\right)^m \right) \\
& \qquad = \det(1+F)\ \det\left( (1+K)(1+L) \right) \\
& \qquad\qquad \times \exp\left( \sum_{m=1}^{n-1} \frac{(-1)^m}{m} \tr \left(K+F+L+KF+KL+FL+KFL\right)^m \right) \\
& \qquad = \det(1+F)\ \textstyle{\det_n} \left( (1+K)(1+L) \right) \exp\left( \sum_{m=1}^{n-1} \frac{(-1)^m}{m} \tr \pi_{m}(K,F,L) \right)
\end{align*}
with
$$
\pi_{m}(K,F,L) := \left(K+F+L+KF+KL+FL+KFL\right)^m - \left(K+L+KL\right)^m \,.
$$
This proves the claimed formula with
$$
p_n(K,L,M) := \sum_{m=1}^{n-1} \frac{(-1)^m}{m} \pi_{m}(K,F,L) \,.
$$
Now we prove the formula in the general case $K,L\in\mathfrak S^n$. We first observe that $\pi_{m}(K,F,L) = ( M+ \text{finite rank})^m - M^m$ is finite rank (with $M=K+L+KL$), so $p_n(K,F,L)$ is finite rank as well. Let $(P_j)$ be a sequence of finite rank projection which converges strongly to the identity and put $K_j=P_jKP_j$ and $L_j=P_j L P_j$. Then
$$
(1+K_j)(1+F)(1+L_j)-1 \to (1+K)(1+F)(1+L) -1
\qquad\text{in}\ \mathfrak S^n
$$
and therefore
$$
\textstyle{\det_n}\left( (1+K_j)(1+F)(1+L_j) \right) \to \textstyle{\det_n}\left( (1+K)(1+F)(1+L) \right) \,.
$$
Similarly,
$$
\textstyle{\det_n}\left( (1+K_j)(1+L_j) \right) \to \textstyle{\det_n}\left( (1+K)(1+L) \right) \,.
$$
Moreover, $K_j\to K$, $L_j\to L$ strongly and $F\in\mathfrak S^1$, so
$$
p_n(K_j,F,L_j) \to p_n(K,F,L) \qquad\text{in}\ \mathfrak S^1
$$
and therefore
$$
e^{\tr p_n(K_j,F,L_j)} \to e^{\tr p_n(K,F,L)} \,.
$$
From this we conclude that the identity with $K$ and $L$ replaced by $K_j$ and $L_j$ implies the corresponding identity for $K$ and $L$. This completes the proof.
\end{proof}



\bibliographystyle{amsalpha}

\end{document}